\documentclass[12pt]{amsart}
\usepackage[margin=1.5in]{geometry}

\usepackage{amssymb, amsmath}
\usepackage{amsthm}
\usepackage{mathrsfs}
\usepackage{cite}
\usepackage{enumitem}
\usepackage{cases}
\usepackage[font={it}, margin=1cm]{caption}
\usepackage{verbatim}
\usepackage{hyperref}

\newtheorem{theorem}{Theorem}[section]
\newtheorem{question}[theorem]{Question}
\newtheorem{conjecture}{Conjecture}
\newtheorem{lemma}[theorem]{Lemma} %%Delete [thm] to re-start numbering
\newtheorem{proposition}[theorem]{Proposition} %%Delete [thm] to re-start
\newtheorem{corollary}[theorem]{Corollary} %%Delete [thm] to re-start
\newtheorem{thmletter}{Theorem}

\newtheorem{corolletter}[thmletter]{Corollary} 

\newtheorem{example}[theorem]{Example}

\newcommand{\p}[1]{\noindent {\newline\bf #1.}}

\newcommand{\aut}{\operatorname{Aut}}

\newcommand{\emo}{\operatorname{End}}

\newcommand{\eq}{\operatorname{Eq}}
\newcommand{\im}{\operatorname{im}}
\newcommand{\rk}{\operatorname{rk}}
\newcommand{\fix}{\operatorname{Fix}}

\newcommand{\SD}{\operatorname{SD}}

\title{The equaliser conjecture for the free group of rank two}
\author{Alan D. Logan}
\address{
Heriot-Watt University,   Edinburgh EH14 4AS,
 Scotland}
 \email{a.logan@hw.ac.uk}
\subjclass[2010]{20E05, 20E07}

\keywords{Free group, equaliser, fixed subgroup.}

\begin{document}
\maketitle

\begin{abstract}
The equaliser of a set of homomorphisms $S: F(a, b)\rightarrow F(\Delta)$ has rank at most two if $S$ contains an injective map, and is not finitely generated otherwise. This proves a strong form of Stallings' Equaliser Conjecture for the free group of rank two.

Results are also obtained for pairs of homomorphisms $g, h:F(\Sigma)\rightarrow F(\Delta)$ when the images are inert in, or retracts of, $F(\Delta)$.
\end{abstract}

\section{Introduction}
The \emph{equaliser} of two free group homomorphisms $g, h:F(\Sigma)\rightarrow F(\Delta)$ is the set of points where they agree, so
$\eq(g, h):=\{x\in F(\Sigma)\mid g(x)=h(x)\}$.
More generally, the equaliser of a set $S:F(\Sigma)\rightarrow F(\Delta)$ of homomorphisms is $\eq(S):=\cap_{g, h\in S}\eq(g, h)$.
If $g$ or $h$ is injective then $\eq(g, h)$ has finite rank, $\rk(\eq(g, h))<\infty$, \cite{Goldstein1986Fixed} and the following conjecture is usually attributed to Stallings\footnote{Stallings' original 1984 version of Conjecture \ref{Qn:StallingsRank} had both maps injective \cite[Problems P1 \& 5]{Stallings1987Graphical}; at this time it was known that $\eq(g, h)$ is finitely generated under this stronger condition \cite{Goldstein1984Automorphisms, Goldstein1985Monomorphisms}, but the case of precisely one injective map was still open.}
\cite[Problem F31]{Baumslag2002Open} \cite[Problem 6]{Dicks1996Group} \cite[Conjecture 8.3]{Ventura2002Fixed}.
\begin{conjecture}[The Equaliser Conjecture]
\label{Qn:StallingsRank}
If $g, h: F(\Sigma)\rightarrow F(\Delta)$ are homomorphisms with $h$ injective then $\rk(\eq(g, h))\leq|\Sigma|$.
\end{conjecture}
This conjecture has its roots in ``fixed subgroups'' $\fix(\phi)$ of free group endomorphisms $\phi: F(\Sigma)\rightarrow F(\Sigma)$ (set $\Sigma=\Delta$, then $\fix(\phi):=\eq(\phi, \operatorname{id})$).
Fixed subgroups have generated a lot of literature from the 1970s onwards
\cite{Bogopolski2016algorithm,
Dyer1975Periodic,
Feighn2018algorithmic,
Gersten1987Fixed,
Jaco1977Surface,
Ventura2002Fixed}.
Indeed, the Equaliser Conjecture has been answered for fixed subgroups: Bestvina and Handel used Thurston's train-track maps to prove that $\rk(\fix(\phi))\leq|\Sigma|$ for $\phi\in\aut(F(\Sigma))$ \cite{Bestvina1992Traintracks}, and Imrich and Turner extended this bound to all endomorphisms \cite{Imrich1989Endomorphisms}. Bergman further extended this bound to all sets of endomorphisms \cite{bergman1999supports}.

Equalisers seem to be harder to understand than fixed subgroups, with only a few papers addressing them
\cite{CMV,
Goldstein1984Automorphisms,
Goldstein1985Monomorphisms,
Goldstein1986Fixed,
Myasnikov2014Post}.
On the other hand, equalisers of free monoid homomorphisms have been studied in computer science for over 70 years, starting with the Post's proof that their triviality is undecidable \cite{Post1946variant} (this is Post's Correspondence Problem).
Many other problems can be easily reduced to this classical problem, such as the mortality problem \cite{Neary2015undecidability} and problems in formal language theory \cite{Harju1997Morphisms}.
That such a fundamental problem is undecidable may be an underlying reason for the relative difficulty in understanding equalisers of free group homomorphisms compared to fixed subgroups, and indeed Post's Correspondence Problem for free groups has recently been discussed as an important open question \cite[Problem 5.1.4]{Dagstuhl2019}.

Our main result considers sets of homomorphisms, much like Bergman's result, and answers the Equaliser Conjecture for the free group of rank two. There are no cardinality restrictions on the set $S$; it may be finite or infinite.
\begin{thmletter}
\label{thm:rankSETS}
Let $S:F(a, b)\rightarrow F(\Delta)$ be a set of homomorphisms, $|S|\geq2$.
\begin{enumerate}[label=(\arabic*)]
\item\label{rankSETS:injONLY} If $S$ contains only injective maps then $\rk(\eq(S))\leq2$.
\item\label{rankSETS:injANDnoninj} If $S$ contains both injective and non-injective maps then $\rk(\eq(S))\leq1$.
\item\label{rankSETS:noninjONLY} If $S$ contains no injective maps then $\eq(S)$ is not finitely generated.
\end{enumerate}
\end{thmletter}

All the possibilities of Theorem \ref{thm:rankSETS} occur; see Example \ref{ex:rankSETS}.

\p{Inert subgroups}
Parts \ref{rankSETS:injANDnoninj} and \ref{rankSETS:noninjONLY} of Theorem \ref{thm:rankSETS} are easily dealt with; the difficulty lies in part \ref{rankSETS:injONLY}. To handle this part we use the following concept:
A subgroup $H$ of a free group $F(\Sigma)$ is \emph{inert} if for all $K\leq F(\Sigma)$ we have $\rk(H\cap K)\leq\rk(K)$.
Inert subgroups were introduced by Dicks and Ventura, who proved that fixed subgroups of monomorphisms are inert \cite[Theorem IV.5.5]{Dicks1996Group}.
Inertness is closed under intersections \cite[Corollary I.4.13]{Dicks1996Group} (see also Lemma \ref{lem:InertiaIntersect}), and so Dicks--Ventura's result extends to fixed subgroups of sets of monomorphisms \cite[Theorem IV.5.7]{Dicks1996Group}.

Examples of inert subgroups of free groups include free factors and retracts (as discussed below),
and every subgroup of rank two \cite{Tardos1992intersection}.
There are inert subgroups which are not fixed subgroups, for example $\langle a, b^2\rangle$ is inert in $F(a, b)$ but not a fixed subgroup as it is not root-closed, while more exotic examples were found by Rosenmann \cite[Example 3.1]{rosenmann2013intersection}.
Recent work on inert subgroups has focused on trying to algorithmically determine inertness by quantifying it \cite{ivanov2018intersection, roy2019degrees} as well as generalising the concept to other groups \cite{wu2014group, Wu2018Fixed, zhang2015fixed}.

Inertness is used to prove our results, and inertness is woven into this topic.
In particular, the following conjecture generalises the Dicks--Ventura result.

\begin{conjecture}
\label{Qn:Inert}
If $S: F(\Sigma)\rightarrow F(\Delta)$ is a set of homomorphisms containing at least one injective map then $\eq(S)$ is inert in $F(\Sigma)$.
\end{conjecture}

Clearly Conjecture \ref{Qn:Inert} implies the Equaliser Conjecture, because we may view $\eq(g, h)$ as $\eq(g, h)\cap F(\Sigma)$ so by inertness $\rk(\eq(g, h))\leq|\Sigma|$.
On the other hand, we prove in \ref{appendix:EquivalentConj} that the conjectures are in fact equivalent (Ventura implies this is so \cite[Conjecture 8.3]{Ventura2002Fixed}).
It is worthwhile emphasising that if
the Equaliser Conjecture
holds for \emph{all} maps then
Conjecture \ref{Qn:Inert}
holds as well for \emph{all} maps, but that if one proves
the Equaliser Conjecture
for some class $\mathcal{C}$ of maps then this does not prove
Conjecture \ref{Qn:Inert}
for $\mathcal{C}$.

The main thrust of this paper is the proof of the following general result on inert subgroups, where part \ref{Inert:1} addresses Conjecture \ref{Qn:StallingsRank} (the Equaliser Conjecture) and part \ref{Inert:2} addresses Conjecture \ref{Qn:Inert}. The condition of the images $\im(g)$ and $\im(h)$ being inert subgroups of $\langle \im(g)\cup\im(h)\rangle$ allows the codomain $F(\Delta)$ to be altered whilst preserving the result.

\begin{thmletter}
\label{thm:Inert}
Let $g, h: F(\Sigma)\rightarrow F(\Delta)$ be injective homomorphisms.
\begin{enumerate}[label=(\arabic*)]
\item\label{Inert:1} If $\im(g)$ is an inert subgroup of $\langle \im(g)\cup\im(h)\rangle$ then $\rk(\eq(g, h))\leq|\Sigma|$.
\item\label{Inert:2} If both $\im(g)$ and $\im(h)$ are inert subgroups of $\langle \im(g)\cup\im(h)\rangle$ then $\eq(g, h)$ is an inert subgroup of $F(\Sigma)$.
\end{enumerate}
\end{thmletter}

This theorem is relevant to Theorem \ref{thm:rankSETS} as Tardos proved that two-generated subgroups of a free group are inert \cite{Tardos1992intersection} (this is a special case of the Hanna Neumann inequality \cite{Friedman2015Hanna, Mineyev2012Submultiplicativity, Mineyev2012Hanna}).

Theorem \ref{thm:Inert}.\ref{Inert:2} may be generalised to sets of homomorphisms; we do this in Corollary \ref{thm:InertSETS}.

\p{Retracts}
A subgroup $H$ of $F(\Delta)$ is a \emph{retract} if there exists a surjection $\rho:F(\Delta)\twoheadrightarrow H$ such that $\rho$ acts as the identity on $H$.
Retracts, like inert subgroups, are important in the theory of fixed subgroups, and in particular they played a key role in Bergman's result on sets of endomorphisms, mentioned above.
Recently, Antol{\'i}n and Jaikin-Zapirain proved that retracts are inert (see Proposition \ref{prop:inertRetracts}).
By adapting the proof of Theorem \ref{thm:Inert}, this allows us to improve the inequality in Theorem \ref{thm:Inert}.\ref{Inert:2} to a strict inequality for retracts.

\begin{thmletter}
\label{thm:Retract}
Let $g, h: F(\Sigma)\rightarrow F(\Delta)$ be injective homomorphisms.
If both $\im(g)$ and $\im(h)$ are retracts of $\langle \im(g)\cup\im(h)\rangle$ and $\im(g)\neq\im(h)$ then $\rk(\eq(g, h))\lneq|\Sigma|$.
\end{thmletter}

The condition in Theorem \ref{thm:Retract} on non-equal images is ruling out the situation where the map $gh^{-1}\in\emo(F(\Sigma))$ (which is defined if $\im(g)=\im(h)$) is an automorphism of $F(\Sigma)$, and here $\eq(g, h)=\fix(gh^{-1})$.

As with Theorem \ref{thm:Inert}, Theorem \ref{thm:Retract} may be generalised to sets of homomorphisms; we do this in Corollary \ref{thm:RetractSETS}.

As well as the Equaliser Conjecture, in 1984 Stallings asked if it is possible to compute the rank of a given equaliser \cite[Problems P3 \& 5]{Stallings1987Graphical}.
Recent results settle this algorithmic problem for certain classes of free group homomorphisms \cite{Bogopolski2016algorithm, Feighn2018algorithmic, ciobanu2020fixed, ciobanu_et_al:LIPIcs:2020:12527}, but despite this progress its solubility remains open in general.
We resolve this problem for retracts.

\begin{thmletter}
\label{thm:RetractAlgorithmic}
There exists an algorithm with input a pair of injective homomorphisms $g, h: F(\Sigma)\rightarrow F(\Delta)$ such that both $\im(g)$ and $\im(h)$ are retracts of $\langle \im(g)\cup\im(h)\rangle$, and output a basis for $\eq(g, h)$.
\end{thmletter}

Again, Theorem \ref{thm:RetractAlgorithmic} may be generalised to sets of homomorphisms; we do this in Corollary \ref{thm:RetractAlgorithmicSETS}.

\p{Inertly induced maps}
A pair of homomorphisms $g, h: F(\Sigma)\rightarrow F(\Delta)$ is \emph{inertly induced} if the pair can be viewed as the restrictions of a pair of homomorphisms $g', h': F(\Sigma')\rightarrow F(\Delta')$ such that $\im(g')$ and $\im(h')$ are inert subgroups of $\langle \im(g')\cup\im(h')\rangle$ (here $\Sigma\subset F(\Sigma')$ and $\Delta\subset F(\Delta')$; see Section \ref{sec:induced} for the formal definition and an example). Our next result follows quickly from Theorem \ref{thm:Inert}, and answers the Equaliser Conjecture for inertly induced pairs.

\begin{corolletter}
\label{corol:InertiaInduced}
Let $g, h: F(\Sigma)\rightarrow F(\Delta)$ be an inertly induced pair of homomorphisms. If $g$ or $h$ is injective then $\rk(\eq(g, h))\leq |\Sigma|$.
\end{corolletter}

Restricting to $|\Sigma'|=2$ gives a stronger result: Define a set of maps $S: F(\Sigma)\rightarrow F(\Delta)$ to be \emph{$F_2$-induced} if the set can be viewed as the restrictions of a set of homomorphisms $S': F(a', b')\rightarrow F(\Delta')$ (see Example \ref{ex:InertlyInduced}).

\begin{corolletter}
\label{corol:F2Induced}
Let $S: F(\Sigma)\rightarrow F(\Delta)$ be an $F_2$-induced set of homomorphisms. If $S$ contains an injection then $\rk(\eq(S))\leq |\Sigma|$.
\end{corolletter}

\p{Outline of the paper}
In Section \ref{sec:SD} we introduce and study the ``stable domain of $g$ with $h$'', $\SD(g, h)$, which is a device for studying the equaliser of two homomorphisms and which generalises the stable image of a free group endomorphism.
In Section \ref{sec:inert} we prove Theorem \ref{thm:Inert}, regarding inertness.
In Section \ref{sec:F2} we prove our main result, Theorem \ref{thm:rankSETS}.
In Section \ref{sec:Retract} we prove Theorems \ref{thm:Retract} and \ref{thm:RetractAlgorithmic}, regarding retracts.
In Section \ref{sec:induced} we prove Corollaries \ref{corol:InertiaInduced} and \ref{corol:F2Induced}, on intertly induced pairs.
Section \ref{sec:Concluding} is a brief discussion on stable domains.
\ref{appendix:EquivalentConj} proves that Conjectures \ref{Qn:StallingsRank} and \ref{Qn:Inert} are equivalent.

\p{Acknowledgements}
I would like to thank Laura Ciobanu for introducing me to the area and for many fruitful discussions about this project,
and Jean-Pierre Mutanguha and the anonymous referee for directing me to the work of Antol{\'i}n--Jaikin-Zapirain, and for a wide variety other extremely helpful comments and suggests for improvement.
This research was supported by EPSRC grant EP/R035814/1.

%%%----------------------------------------------%%%
%%%------------The stable domain------------%%%
%%%----------------------------------------------%%%

\section{Equalisers as fixed subgroups}
\label{sec:SD}
In this section we view equalisers as fixed subgroups, as explained below. This view may alter the rank of the ambient free group, but crucially the rank is preserved under the assumptions of Theorem \ref{thm:Inert}, and for the free group of rank two.

We begin with a lemma which, under very specific conditions, allows us to view equalisers as fixed subgroups.
If $g, h: F(\Sigma)\rightarrow F(\Delta)$ are homomorphisms with $h$ injective and $\im(g)\leq\im(h)$ then we may define:
\begin{align*}
\psi_{(g, h)}: F(\Sigma)&\rightarrow F(\Sigma)\\
x&\mapsto h^{-1}(g(x))
\end{align*}
Here we can apply $h^{-1}$ to $g(x)$ as $\im(g)\leq\im(h)$, and the map $\psi_{(g, h)}$ is a function (hence homomorphism) as $h$ is injective.

\begin{lemma}
\label{lem:FixVsEq}
Let $g, h: F(\Sigma)\rightarrow F(\Delta)$ be homomorphisms with $h$ injective and $\im(g)\leq\im(h)$. Then $\eq(g, h)=\fix(\psi_{(g, h)})$.
\end{lemma}

\begin{proof}
Clearly $\fix(\psi_{(g, h)})\leq\eq(g, h)$, while $\eq(g, h)\leq\fix(\psi_{(g, h)})$ since if $g(x)=h(x)$ then $h^{-1}(g(x))=x$, as $h$ is injective, and clearly $x\in h^{-1}(\im(g))$.
\end{proof}

The goal of this section is to take two maps and restrict their domain in such a way that we may apply Lemma \ref{lem:FixVsEq} to understand their equaliser.

\p{The stable domain}
For endomorphisms $\phi: F\rightarrow F$ the \emph{stable image} of $\phi$ is $\phi^{\infty}(F):=\cap_{i=0}^{\infty}\phi^i(F)$. Imrich and Turner used this gadget to prove that $\rk(\fix(\phi))\leq\rk(F)$ \cite{Imrich1989Endomorphisms}.
We now generalise this construction to equalisers.

Let $g, h: F(\Sigma)\rightarrow F(\Delta)$ be homomorphisms. Define $H_0=F(\Sigma)$, and inductively define $H_{i+1}=g^{-1}(g(H_i)\cap h(H_i))$. Then define
\[
\SD(g, h):=\cap_{i=0}^{\infty}H_i.
\]
We call $\SD(g, h)$ the \emph{stable domain of $g$ with $h$}.
(The construction is not symmetric; see Example \ref{ex:symmetry}.)
The name ``stable domain'' is because we can use the restrictions $g|_{\SD(g, h)}$ and $h|_{\SD(g, h)}$ to understand $\eq(g, h)$ (see Lemma \ref{lem:SDandEq}).
By taking $\Sigma=\Delta$ and $g$ to be the identity map, we see that the stable image is a special case of the stable domain.

We start by characterising stable domains. The proof of the lemma uses an inductive argument, and the same basic argument is used in many of our proofs below.
In the following we mean maximal with respect to inclusion.

\begin{lemma}
\label{lem:CompareImages}
Let $h$ be injective. Then $\SD(g, h)$ is the maximal subgroup $K$ of $F(\Sigma)$ such that $g(K)\leq h(K)$.
\end{lemma}

\begin{proof}
We first prove that $g({\SD(g, h)})\leq h({\SD(g, h)})$.
So, let $x\in \SD(g, h)$. Then $x\in H_i$ for all $i\geq0$. Hence, for all $j\geq1$ there exists some $y_j\in g(H_j)\cap h(H_j)$ such that $x\in g^{-1}(y_j)$. Then $g(x)=y_j\in h(H_j)$, and so $g(x)\in h(H_j)$ for all $j\geq0$. Hence, $g(x)\in \cap h(H_j)$. By injectivity of $h$, we have $\cap h(H_j)=h(\cap H_j)=h({\SD(g, h)})$, and so $g(x)\in h({\SD(g, h)})$ as required.

For maximality, suppose $K\leq F(\Sigma)$ is such that $g(K)\leq h(K)$.
Clearly $K\leq H_0=F(\Sigma)$. If $K\leq H_i$ then $g(K)\leq g(H_i)\cap h(H_i)$, and so $K\leq g^{-1}(g(H_i)\cap h(H_i))=H_{i+1}$. Hence, by induction $K\leq H_i$ for all $i\geq0$, and so $K\leq\cap H_i=\SD(g, h)$ as required.
\end{proof}

We now give two examples of stable domains. Our first example shows that $\SD(g, h)\neq\SD(h, g)$ in general, even if both maps are injective. Later, in Proposition \ref{prop:SDsymmetry}, we  classify when $\SD(g, h)=\SD(h, g)$ for $g$, $h$ injective.

\begin{example}
\label{ex:symmetry}
Define $g:F(x, y)\rightarrow F(a, b)$ by $g: x\mapsto a^2, y\mapsto b$ and $h: x\mapsto a, y\mapsto b^2$. As $g(\langle x\rangle)\leq h(\langle x\rangle)$, we have that $x\in\SD(g, h)$ by Lemma \ref{lem:CompareImages}, and similarly $y\in\SD(h, g)$.
On the other hand, $y\not\in\SD(g, h)$ and $x\not\in\SD(h, g)$, as $\im(g)\cap\im(h)$ is a proper subgroup of both $\im(g)$ and $\im(h)$, and so neither stable domain is the whole of $F(x, y)$. Hence, $\SD(g, h)\neq\SD(h, g)$.
\end{example}

Later, in Theorem \ref{thm:EqRankBoundedSD}, we see that it is important to understand the rank of the stable domain $\SD(g, h)$ when $h$ is injective. Unfortunately, as our next example shows, stable domains are not necessarily finitely generated under this restriction.

\begin{example}
\label{ex:finiteGen}
Define $g: x\mapsto ab, y\mapsto1$ and $h: x\mapsto a^2, y\mapsto b^2$. Then $\im(g)\cap\im(h)$ is trivial, and so the subgroup $H_1$ in the definition of the stable domain is $\ker(g)$. We then see inductively that $H_i=\ker(g)$ for all $i\geq0$, and so $\SD(g, h)=\ker(g)$. As $\ker(g)$ is a normal subgroup of infinite index in $F(\Sigma)$, we have that $\SD(g, h)=\ker(g)$ is not finitely generated.
\end{example}

The above example is in fact extreme, as $\eq(g, h)$ is not finitely generated while $\eq(h, g)$ is trivial (because $h$ is injective and $\im(g)\cap\im(h)$ is trivial).

\p{Equalisers as fixed subgroups}
We now explain how to use stable domains to view equalisers as fixed subgroups.
Firstly, we can use the restrictions $g|_{\SD(g, h)}$ and $h|_{\SD(g, h)}$ to understand $\eq(g, h)$.

\begin{lemma}
\label{lem:SDandEq}
$\eq(g, h)=\eq(g|_{\SD(g, h)}, h|_{\SD(g, h)})$.
\end{lemma}

\begin{proof}
Clearly $\eq(g|_{\SD(g, h)}, h|_{\SD(g, h)})\leq \eq(g, h)$. For the other direction we prove that $\eq(g, h)\leq \SD(g, h)$, which is sufficient.
So, let $x\in \eq(g, h)$. Then $x\in H_0$, while if $x\in H_i$ then $g(x)=h(x)\in g(H_i)\cap h(H_i)$, and so $x\in g^{-1}(g(h_i)\cap h(H_i))=H_{i+1}$. Therefore, by induction we have that $x\in H_i$ for all $i\geq0$, and so $x\in \cap H_i=\SD(g, h)$ as required.
\end{proof}

If $h$ is injective then we can define $\psi_{(g|_{\SD(g, h)}, h|_{\SD(g, h)})}\in \emo(\SD(g, h))$ as in Lemma \ref{lem:FixVsEq}.
Crucially, $\eq(g, h)$ is the set of fixed points of this map.

\begin{lemma}
\label{lem:EqualAsFixed}
Let $g, h: F(\Sigma)\rightarrow F(\Delta)$ be homomorphisms with $h$ injective. Then the endomorphism $\phi_{(g, h)}:=\psi_{(g|_{\SD(g, h)}, h|_{\SD(g, h)})}\in \emo(\SD(g, h))$ satisfies $\eq(g, h)=\fix(\phi_{(g, h)})$.
\end{lemma}

\begin{proof}
By Lemma \ref{lem:CompareImages}, the maps $g|_{\SD(g, h)}$ and $h|_{\SD(g, h)}$ satisfy the conditions of Lemma \ref{lem:FixVsEq}, and so $\phi_{(g, h)}$ is an endomorphism of $\SD(g, h)$
which satisfies $\eq(g|_{\SD(g, h)}, h|_{\SD(g, h)})=\fix(\phi_{(g, h)}).$
The result then follows by Lemma \ref{lem:SDandEq}.
\end{proof}

Combining Lemma \ref{lem:EqualAsFixed} with known results on fixed subgroups of free groups, we have the following.
\begin{theorem}
\label{thm:EqRankBoundedSD}
Let $g, h: F(\Sigma)\rightarrow F(\Delta)$ be homomorphisms.
\begin{enumerate}[label=(\alph*)]
\item\label{EqRankBoundedSD:1} If $h$ is injective then $\rk(\eq(g, h))\leq\rk(\SD(g, h))$.
\item\label{EqRankBoundedSD:2} If both $g$ and $h$ are injective then $\eq(g, h)$ is inert in $\SD(g, h)$.
\end{enumerate}
\end{theorem}

\begin{proof}
Suppose $h$ is injective. Consider the map $\phi_{(g, h)}\in\emo(\SD(g, h))$ from Lemma \ref{lem:EqualAsFixed}, with $\fix(\phi_{(g, h)})=\eq(g, h)$. Then $\rk(\eq(g, h))=\rk(\fix(\phi_{(g, h)}))\leq\rk(\SD(g, h))$ \cite{Imrich1989Endomorphisms}, as required.

Suppose both $g$ and $h$ are injective. Recalling that $\phi_{(g, h)}\in\emo(\SD(g, h))$ is defined by $x\mapsto h^{-1}(g(x))$, as $g$ is injective the map $\phi_{(g, h)}$ is also injective. Hence, $\fix(\phi_{(g, h)})$ is inert in $\SD(g, h)$ \cite[Theorem IV.5.5]{Dicks1996Group}. The result follows as $\eq(g, h)=\fix(\phi_{(g, h)})$, by Lemma \ref{lem:EqualAsFixed}.
\end{proof}

In order to apply Theorem \ref{thm:EqRankBoundedSD} to the Equaliser Conjecture we would need to show that if $h$ is injective then the rank of $\SD(g, h)$ is bounded by $|\Sigma|$.
By Example \ref{ex:finiteGen}, this is not true in general.

%%%----------------------------------------------%%%
%%%-----------------Inertia---------------------%%%
%%%----------------------------------------------%%%

\section{Proof of Theorem \ref{thm:Inert}}
\label{sec:inert}
In this section we prove Theorem \ref{thm:Inert}; we do this by using the assumed conditions regarding inertness to understand $\SD(g, h)$, and then apply Theorem \ref{thm:EqRankBoundedSD}. We also extend Theorem \ref{thm:Inert}.\ref{Inert:2} to cover sets of homomorphisms, rather than just pairs.

We first record the following result of Takahasi which we use frequently below \cite{Takahasi1951chains} (see also \cite[Theorem 15.2]{Kapovich2002stallings}).
This was applied by Imrich and Turner in order to prove that $\rk(\fix(\phi))\leq|\Sigma|$ for all $\phi\in\emo(F(\Sigma))$ \cite{Imrich1989Endomorphisms}.

\begin{proposition}[Takahasi, 1951 \cite{Takahasi1951chains}]
\label{prop:Exercise}
Let $K_0\geq K_1\geq\ldots$ be a nested sequence of free groups, and write $K:=\cap_{i=0}^{\infty}K_i$. If there exists some $n\in\mathbb{N}$ such that $\rk(K_i)\leq n$ for all $i\geq0$ then $K$ is a free factor of all but finitely many of the $K_i$. In particular, $\rk(K)\leq n$.
\end{proposition}

If $g$ is injective then the subgroups $H_i$ from the definition of the stable domain, so where $H_0=F(\Sigma)$ and $H_{i+1}=g^{-1}(g(H_i)\cap h(H_i))$, form a nested sequence of free groups $H_0\geq H_1\geq\ldots$ and so Proposition \ref{prop:Exercise} is applicable. Our first lemma corresponds to Theorem \ref{thm:Inert}.\ref{Inert:1}.

\begin{lemma}
\label{lem:Inert:1}
Let $g, h: F(\Sigma)\rightarrow F(\Delta)$ be injective homomorphisms. If $\im(g)$ is an inert subgroup of $\langle \im(g)\cup\im(h)\rangle$ then $\rk(\eq(g, h))\leq\rk(\SD(g, h))\leq|\Sigma|$.
\end{lemma}

\begin{proof}
We first prove that for all $i\geq0$ we have $g(H_{i+1})=\im(g)\cap h(H_i)$, with $H_i$ the subgroups in the definition of the stable domain. As $\im(g)=g(H_0)$, this holds for $i=0$. If the result holds for $i$ then we have the following, with the last line obtained as $H_i\leq H_{i-1}$ (as $g$ is injective) so $h(H_i)\leq h(H_{i-1})$:
\begin{align*}
g(H_{i+1})
&=g(H_i)\cap h(H_i)\\
&=\im(g)\cap h(H_{i-1})\cap h(H_i)\\
&=\im(g)\cap h(H_i)
\end{align*}
Therefore, by induction we have that $g(H_{i+1})=\im(g)\cap h(H_i)$. It then follows by induction, and applying the fact that $g$ is injective and $\im(g)$ is inert, that $\rk(H_i)\leq |\Sigma|$ for all $i\geq0$.
Then the sequence
\[
H_0\geq H_1\geq H_2\geq\ldots
\]
is nested, and by the above each term has rank at most $|\Sigma|$; it follows from Proposition \ref{prop:Exercise} that $\rk(\SD(g, h))=\rk\left(\cap_{i=0}^{\infty}H_i\right)\leq |\Sigma|$. The bound on $\rk(\eq(g, h))$ then follows from Theorem \ref{thm:EqRankBoundedSD}.\ref{EqRankBoundedSD:1}.
\end{proof}

Theorem \ref{thm:Inert}.\ref{Inert:2} uses the fact that inertness is stable under intersections.
This was first proven by Dicks--Ventura using the language of $F$-sets \cite[Corollary I.4.13]{Dicks1996Group}.
For the sake of completeness, we give a low-level proof now.

\begin{lemma}
\label{lem:InertiaIntersect}
Let $\{A_i\}_I$, $I\subset\mathbb{N}$, be a set of inert subgroups of a free group $F(\Sigma)$. Then $\cap A_i$ is inert.
\end{lemma}

\begin{proof}
Suppose $A$ and $B$ are inert, and let $K\leq F(\Sigma)$ be arbitrary.
Then $\rk(A\cap B\cap K)\leq\rk(B\cap K)$ as $A$ is inert, while $\rk(B\cap K)\leq\rk(K)$ as $B$ is inert. Hence, $A\cap B$ is inert and so the result holds for finite sets $\{A_i\}_{I}$.

Suppose $\{A_i\}_I$ is a countable set of inert subgroups, and let $K$ be an arbitrary subgroup. We may suppose $I=\mathbb{N}$, and so define $B_n:=\cap_{i=0}^n A_i$. Then $B_n$ is inert, by the above.
Now, the sequence
\[
(B_0\cap K)\geq (B_1\cap K)\geq (B_2\cap K)\geq\ldots
\]
is nested, and by inertness each term has rank at most $\rk(K)$; it follows from Proposition \ref{prop:Exercise} that $\rk\left(\cap_{i=0}^{\infty}(B_i\cap K)\right)\leq \rk(K)$. The result then follows as $\cap_{i=0}^{\infty}(B_i\cap K)=\left(\cap_{i=0}^{\infty}A_i\right)\cap K$.
\end{proof}

The following lemma corresponds to Theorem \ref{thm:Inert}.\ref{Inert:2}. The proof uses the easy fact that inertness is transitive: if we have a chain of subgroups $A< B< C$ with $A$ inert in $B$ and $B$ inert in $C$ then $A$ is inert in $C$.

\begin{lemma}
\label{lem:Inert:2}
Let $g, h: F(\Sigma)\rightarrow F(\Delta)$ be injective homomorphisms.
If both $\im(g)$ and $\im(h)$ are inert subgroups of $\langle \im(g)\cup\im(h)\rangle$ then $\eq(g, h)$ is an inert subgroup of $F(\Sigma)$.
\end{lemma}

\begin{proof}
We shall write $J_{g, h}:=\langle \im(g)\cup\im(h)\rangle$. By Theorem \ref{thm:EqRankBoundedSD}.\ref{EqRankBoundedSD:2} and the transitivity of inertness, it is sufficient to prove that $\SD(g, h)$ is inert in $F(\Sigma)$. To do this we consider the subgroups $H_i$ from the definition of the stable domain.
Note that $H_0=F(\Sigma)$ is inert in $F(\Sigma)$.
Suppose $H_i$ is inert in $F(\Sigma)$. Then $g(H_i)$ is inert in $\im(g)$ which is inert in $J_{g, h}$, and so by transitivity we have that $g(H_i)$ is inert in $J_{g, h}$. Similarly, $h(H_i)$ is inert in $J_{g, h}$. Hence, by Lemma \ref{lem:InertiaIntersect}, $g(H_i)\cap h(H_i)$ is inert in $J_{g, h}$ and so is inert in $\im(g)\leq J_{g, h}$. As $g$ is injective its inverse $g^{-1}: \im(g)\rightarrow F(\Sigma)$ is an isomorphism and so $H_{i+1}:=g^{-1}(g(H_i)\cap h(H_i))$ is inert in $F(\Sigma)$.
It follows by induction that $H_i$ is inert in $F(\Sigma)$ for all $i\geq0$. Hence, $\SD(g, h)$ is inert in $F(\Sigma)$ by Lemma \ref{lem:InertiaIntersect}.
The result follows.
\end{proof}

We now prove Theorem \ref{thm:Inert}.

\begin{proof}[Proof of Theorem \ref{thm:Inert}]
This follows immediately from Lemmas \ref{lem:Inert:1} and \ref{lem:Inert:2}.
\end{proof}

Using Lemma \ref{lem:InertiaIntersect} we can generalise Theorem \ref{thm:Inert}.\ref{Inert:2} as follows. For a set $S: F(\Sigma)\rightarrow F(\Delta)$ of homomorphisms, define $\Gamma_S$ to be the graph with vertex set $S$, with an edge connecting $g, h\in S$ if $\im(g)$ and $\im(h)$ are inert in $\langle\im(g)\cup\im(h)\rangle$.

\begin{corollary}
\label{thm:InertSETS}
Let $S: F(\Sigma)\rightarrow F(\Delta)$ be a set of injective homomorphisms such that the graph $\Gamma_S$ is connected. Then $\eq(S)$ is an inert subgroup of $F(\Sigma)$.
\end{corollary}

\begin{proof}
As $\Gamma_S$ is connected, $\eq(S)$ is the intersection of those equalisers $\eq(g, h)$ such that there is a edge connecting $g$ and $h$. By Theorem \ref{thm:Inert}.\ref{Inert:2}, each such equaliser is inert, and so $\eq(S)$ is inert by Lemma \ref{lem:InertiaIntersect}.
\end{proof}

%%%----------------------------------------------%%%
%%%------------Rank two case----------------%%%
%%%----------------------------------------------%%%

\section{The free group of rank two}
\label{sec:F2}

We are now able to prove Theorem \ref{thm:rankSETS}, which describes the rank of $\eq(S)$ for $S: F(a, b)\rightarrow F(\Delta)$ a set of homomorphisms.

\begin{proof}[Proof of Theorem \ref{thm:rankSETS}]
For part \ref{rankSETS:injONLY}, suppose every element of $S$ is injective. Two-generated subgroups of free groups are inert, by the Hanna Neumann inequality, and so the conditions of Theorem \ref{thm:Inert}.\ref{Inert:2} are satisfied for all $g, h\in S$. Therefore, for all $g, h\in S$ we have that $\eq(g,h)$ is inert in $F(a, b)$. Then, as $S$ is necessarily countable, Lemma \ref{lem:InertiaIntersect} gives us that $\eq(S)$ is inert in $F(a, b)$. Hence, $\rk(\eq(S))\leq2$ as required.

For part \ref{rankSETS:injANDnoninj}, suppose $S$ contains both injective and non-injective maps. Let $g\in S$ be injective, and note that $\eq(S)=\cap_{h\in S}\eq(g, h)$. Let $h\in S$ be non-injective. Then $\eq(g, h)\leq g^{-1}(\im(g)\cap\im(h))$, while $g^{-1}(\im(g)\cap\im(h))$ is cyclic because $\im(h)$ is cyclic and because $g$ is injective. Hence, $\rk(\eq(g, h))\leq1$ and so as $\eq(S)\leq\eq(g, h)$ we have that $\rk(\eq(S))\leq1$ as required.

For part \ref{rankSETS:noninjONLY}, suppose $S$ does not contain an injection. Then $\eq(S)$ is a normal subgroup of $F(a, b)$:
Consider $x\in\eq(S)$ and let $y\in F(a, b)$, then for all $g\in S$ we have $g(y^{-1}xy)=g(y^{-1})g(x)g(y)=g(x)$, as $\im(g)$ is cyclic, and so $y^{-1}xy\in\eq(S)$ as required.
If $\eq(S)$ is a normal subgroup of finite index $n$, say, then for all $x\in F(a, b)$ and all $g, h\in S$ we have that $g(x^n)=h(x^n)$ and so, as roots are unique in free groups, $g(x)=h(x)$; this says that $|S|=1$, which is impossible as $|S|\geq2$.
Therefore, $\eq(S)$ is a normal subgroup of infinite index, so is either trivial or not finitely generated.
However, $\eq(S)$ is non-trivial as $[a, b]\in\eq(S)$ because $[a, b]\in\ker(g)$ for all $g\in S$, and the result follows.
\end{proof}

We now give examples of sets of maps which show that $\eq(S)$ can have any of the possible ranks in parts \ref{rankSETS:injONLY} and \ref{rankSETS:injANDnoninj} of Theorem \ref{thm:rankSETS}.

\begin{example}
\label{ex:rankSETS}
Take $\Delta=\{x, y\}$.
We start with injective maps, as in part \ref{rankSETS:injONLY} of the theorem.
If $g(a)=x$, $g(b)=y$ and $h(a)=y$, $h(b)=x$ then both maps are injective and $\eq(g, h)=1$.
If $g(a)=x$, $g(b)=y$ and $h(a)=x$, $h(b)=y^{-1}$ then both maps are injective and $\eq(g, h)=\langle a\rangle$.
If $g(a)=xy$, $g(b)=y$ and $h(a)=x$, $h(b)=y$ then both maps are injective, while $b, aba^{-1}\in\eq(g, h)$ so $\eq(g, h)$ is non-abelian and hence has rank two.

We now take one injective and one non-injective map, as in part \ref{rankSETS:injANDnoninj} of the theorem.
If $g(a)=x$, $g(b)=y$ and $h(a)=y$, $h(b)=1$ then $g$ is injective, $h$ is non-injective and $\eq(g, h)=1$.
If $g(a)=x$, $g(b)=y$ and $h(a)=x$, $h(b)=1$ then $g$ is injective, $h$ is non-injective and $\eq(g, h)=\langle x\rangle$.
\end{example}

%%%----------------------------------------------%%%
%%%------------Retracts------------------------%%%
%%%----------------------------------------------%%%

\section{Retracts}
\label{sec:Retract}
In this section we prove structural and algorithmic results when one or both of the maps $g, h: F(\Sigma)\rightarrow F(\Delta)$ has image which is a retract of $\langle\im(g), \im(h)\rangle$. In particular, we prove Theorems \ref{thm:Retract} and \ref{thm:RetractAlgorithmic}.

\p{Structural results}
We begin by proving that retracts are inert, and hence Theorem \ref{thm:Inert} is applicable.
The proof of this result is due to Antol\'in--Jaikin-Zapirain, who introduced and studied \emph{$L^2$-invariant subgroups}.
They remarked that their results for surface groups imply that retracts of free groups are inert, and we explain here how to extract this from their article.
Our proof uses transitivity of retracts: if we have $A< B< C$ with $A$ a retract of $B$ and $B$ a retract of $C$ then $A$ is a retract of $C$ (compose the retraction maps).

\begin{proposition}
\label{prop:inertRetracts}
If $H$ is a retract of the free group $F$ then $H$ is inert in $F$.
\end{proposition}

\begin{proof}
There exists an embedding $\phi: F\hookrightarrow S_g$ of $F$ into a surface group such that $\phi(F)$ is a retract of $S_g$.
By transitivity of retracts, $\phi(H)$ is a retract of $S_g$.
Hence, $\phi(H)$ is $L^2$-invariant in $S_g$ \cite[Proposition 4.5]{antolin2020hanna}. Therefore, $\phi(H)$ is inert in $S_g$ \cite[Proposition 5.2]{antolin2020hanna}, and thus inert in $\phi(F)$.
The result follows.
\end{proof}

We therefore have the following corollary of Theorem \ref{thm:Inert}.

\begin{corollary}
\label{corol:Retract}
Let $g, h: F(\Sigma)\rightarrow F(\Delta)$ be injective homomorphisms.
\begin{enumerate}[label=(\arabic*)]
\item\label{retract:1} If $\im(g)$ is a retract of $\langle \im(g)\cup\im(h)\rangle$ then $\rk(\eq(g, h))\leq|\Sigma|$.
\item\label{retract:2} If both $\im(g)$ and $\im(h)$ are retracts of $\langle \im(g)\cup\im(h)\rangle$ then $\eq(g, h)$ is an inert subgroup of $F(\Sigma)$.
\end{enumerate}
\end{corollary}

To improve on this corollary we give the following lemma, which re-visits the proof of Lemma \ref{lem:Inert:2}. If we take both classes of subgroups to be retracts then the hardest condition to verify is condition \ref{Inert:2:4}, which is a deep result of Bergman \cite[Lemma 18]{bergman1999supports}.

\begin{lemma}
\label{lem:Inert:Generalisations}
Let $g, h: F(\Sigma)\rightarrow F(\Delta)$ be injective homomorphisms, and let $\mathcal{C}_{\Sigma}$ and $\mathcal{C}_{\Delta}$ be classes of subgroups of $F(\Sigma)$ and $F(\Delta)$, respectively, such that:
\begin{enumerate}[label=(\alph*)]
\item\label{Inert:2:1}
$F(\Sigma)\in\mathcal{C}_{\Sigma}$,
\item\label{Inert:2:2}
if $A\in\mathcal{C}_{\Sigma}$ then $g(A), h(A)\in\mathcal{C}_{\Delta}$,
\item\label{Inert:2:4}
if $A, B\in\mathcal{C}_{\Delta}$ then $A\cap B\in\mathcal{C}_{\Delta}$,
\item\label{Inert:2:3}
if $A\in\mathcal{C}_{\Delta}$ and $A\leq\im(g)\cap\im(h)$ then $g^{-1}(A)\in\mathcal{C}_{\Sigma}$.
\end{enumerate}
The subgroups $H_0\geq H_1\geq\ldots$ of $F(\Sigma)$ in the definition of $\SD(g, h)$ are all contained in $\mathcal{C}_{\Sigma}$.
\end{lemma}

\begin{proof}
Consider the subgroups $H_i$ in the definition of the stable domain.
By \ref{Inert:2:1}, $H_0=F(\Sigma)\in \mathcal{C}_{\Sigma}$.
Now, if $H_i\in\mathcal{C}_{\Sigma}$ then $g(H_i), h(H_i)\in\mathcal{C}_{\Delta}$ by \ref{Inert:2:2}.
Hence, $g(H_i)\cap h(H_i)\in\mathcal{C}_{\Delta}$, by \ref{Inert:2:4}, and so $H_{i+1}:=g^{-1}(g(H_i)\cap h(H_i))\in\mathcal{C}_{\Sigma}$ by \ref{Inert:2:3}.
It follows by induction that $H_i\in\mathcal{C}_{\Sigma}$ for all $i\geq0$.
\end{proof}

A \emph{proper retract} of a group is a proper subgroup which is a retract.
Note that if $H$ is a proper retract of a free group $F(\Delta)$ then $\rk(H)\lneq\rk(F(\Delta))$, as $F(\Delta)$ surjects onto $H$ and free groups are Hopfian.
The following lemma covers $H_i$ with $i\geq1$ but not $H_0$ because $H_0=F(\Sigma)$ is not a proper retract of $F(\Sigma)$.
Note that if we have $A< B< C$ with $A$ a retract of $C$ then $A$ is a retract of $B$ (restrict the retraction map to $B$).

\begin{lemma}
\label{lem:retractSD}
Let $g, h: F(\Sigma)\rightarrow F(\Delta)$ be injective homomorphisms.
If both $\im(g)$ and $\im(h)$ are retracts of $\langle \im(g)\cup\im(h)\rangle$ and $\im(g)\neq\im(h)$ then
\begin{enumerate}
\item\label{retractSD:1} the subgroups $H_1\geq H_2\geq\ldots$ of $F(\Sigma)$ in the definition of $\SD(g, h)$, and
\item\label{retractSD:2} the stable domain $\SD(g, h)$
\end{enumerate}
are all proper retracts of $F(\Sigma)$.
\end{lemma}

\begin{proof}
Let $g, h: F(\Sigma)\rightarrow F(\Delta)$ be as in the statement of the lemma.
We first prove that the $H_i$ are retracts of $F(\Sigma)$, and to do this it is sufficient to prove that if we take $\mathcal{C}_{\Sigma}$ and $\mathcal{C}_{\Delta}$ to be the retracts of $F(\Sigma)$ and $\langle \im(g)\cup\im(h)\rangle$ respectively then these satisfy the conditions of Lemma \ref{lem:Inert:Generalisations}.
So, condition \ref{Inert:2:1} of the lemma immediately holds, while \ref{Inert:2:4} in known to hold \cite[Lemma 18]{bergman1999supports}.
For condition \ref{Inert:2:2}, if $A$ is a retract of $F(\Sigma)$ then $g(A)$ is a retract of $\im(g)$, and so, by transitivity, $g(A)$ is a retract of $\langle \im(g)\cup\im(h)\rangle$. Similarly, $h(A)$ is a retract of $\langle \im(g)\cup\im(h)\rangle$, so \ref{Inert:2:2} holds.
For condition \ref{Inert:2:3},
suppose $A$ is a retract of $\langle\im(g)\cup\im(h)\rangle$ and $A\leq\im(g)\cap\im(h)$.
Then we have $A\leq\im(g)\leq\langle\im(g)\cup\im(h)\rangle$, and so $A$ is a retract of $\im(g)$. As $g$ is injective, we have that $g^{-1}(A)$ is a retract of $F(\Sigma)$, as required. Hence, the $H_i$ are retracts of $F(\Sigma)$, as claimed.

As $\im(g)\neq\im(h)$, we have that $H_0\gneq H_1$, and so every $H_i$ for $i\geq1$ is a proper retract of $F(\Sigma)$. Hence, \eqref{retractSD:1} holds. For \eqref{retractSD:2}, recall that $\SD(g, h):=\cap H_i$. By Proposition \ref{prop:Exercise}, $\SD(g, h)$ is a free factor of all but finitely many of the $H_i$ and so is a retract of some proper retract of $F(\Sigma)$. The result follows.
\end{proof}

We can now prove Theorem \ref{thm:Retract}.

\begin{proof}[Proof of Theorem \ref{thm:Retract}]
Combining the inequalities $\rk(\eq(g, h))\leq\rk(\SD(g, h))$ and $\rk(\SD(g, h))\lneq |\Sigma|$ gives the result; these hold by Theorem \ref{thm:EqRankBoundedSD} and Lemma \ref{lem:retractSD} respectively.
\end{proof}

We now generalise Theorem \ref{thm:Retract} to sets of homomorphisms.
Recall from Corollary \ref{thm:InertSETS} that for a set $S: F(\Sigma)\rightarrow F(\Delta)$ of homomorphisms, $\Gamma_S$ is the graph with vertex set $S$, with an edge connecting $g, h\in S$ if $\im(g)$ and $\im(h)$ are inert in $\langle\im(g)\cup\im(h)\rangle$.

\begin{corollary}
\label{thm:RetractSETS}
Let $S: F(\Sigma)\rightarrow F(\Delta)$ be a set of injective homomorphisms such that the graph $\Gamma_S$ is connected, and such that there exists some $g,h\in S$ such that $\im(g)$ and $\im(h)$ are retracts of $\langle \im(g)\cup\im(h)\rangle$ and $\im(g)\neq\im(h)$. Then $\rk(\eq(S))\lneq|\Sigma|$.
\end{corollary}

\begin{proof}
By Theorem \ref{thm:Retract}, if $g, h\in S$ are the maps in the statement of the corollary then $\rk(\eq(g, h))\lneq|\Sigma|$.
By Corollary \ref{thm:InertSETS}, $\eq(S)$ is an inert subgroup of $F(\Sigma)$, and so as $\eq(S)=\eq(S)\cap\eq(g, h)$ we have $\rk(\eq(S))=\rk(\eq(S)\cap\eq(g, h))\leq\rk(\eq(g, h))\lneq|\Sigma|$, as required.
\end{proof}

\p{Algorithmic results}
Our first algorithmic result is the computation of the stable domain of $g$ with $h$.
We store free group homomorphisms in terms of images of basis elements, so for example the maps $g, h: F(\Sigma)\rightarrow F(\Delta)$ are stored as lists $(g(x_i)\mid x_i\in\Sigma)$ and $(h(x_i)\mid x_i\in\Sigma)$.

\begin{lemma}
\label{lem:algorithmicSD}
There exists an algorithm with input a pair of injective homomorphisms $g, h: F(\Sigma)\rightarrow F(\Delta)$ such that both $\im(g)$ and $\im(h)$ are retracts of $\langle \im(g)\cup\im(h)\rangle$, and output a basis for $\SD(g, h)$.
\end{lemma}

\begin{proof}
If $\im(g)=\im(h)$ then $\SD(g, h)=F(\Sigma)$ and there is nothing to prove. Therefore, suppose that $\im(g)\neq\im(h)$.

A basis for each of the subgroups $H_i$ in the definition of $\SD(g, h)$ can be constructed via standard algorithms.
In particular, we can compute a basis for $g(H_i)\cap h(H_i)$ \cite[Corollary 9.5]{Kapovich2002stallings}, and as $g$ is injective we can then compute a basis for $H_{i+1}=g^{-1}(g(H_i)\cap h(H_i))$.

Now, we have a descending chain $H_0\geq H_1\geq H_2\geq\ldots$ where by Lemma \ref{lem:retractSD} each $H_i$ is a retract of $H_0=F(\Sigma)$, and so $H_{i+1}$ is a retract of $H_i$.
If $H_{i+1}$ is a proper retract of $H_i$ then $\rk(H_i)\lneq\rk(H_{i+1})$.
If $H_{i+1}$ is not a proper retract of $H_i$ then $H_i=H_{i+1}$ and by the construction of these subgroups the chain stabilises, that is $H_i=H_j$ for all $j\geq i$.
By the definition of $\SD(g, h)$, if $H_i=H_{i+1}$ then $H_i=\SD(g, h)$.
Thus, each step in the chain $H_0\geq H_1\geq\ldots$ reduces the rank of the $H_i$ until the chain stabilises, and so it stabilises in at most $|\Sigma|$ steps. Moreover, it stabilises at the subgroup $\SD(g, h)$.

The algorithm to compute a basis for $\SD(g, h)$ is therefore as follows.
Compute bases for the subgroups $H_1, H_2, \ldots$ and at each step determine whether $H_i=H_{i+1}$ (we can check this equality via a standard algorithm \cite[Proposition 7.2]{Kapovich2002stallings}).
Stop when $H_i=H_{i+1}$ or when $i=|\Sigma|$, and output the basis for $H_i$ as a basis for $\SD(g, h)$.
\end{proof}

We can now prove Theorem \ref{thm:RetractAlgorithmic}, on computing a basis for $\eq(g, h)$.

\begin{proof}[Proof of Theorem \ref{thm:RetractAlgorithmic}]
We first establish three facts. Firstly, the image of the endomorphism $\phi_{(g, h)}$ of $\SD(g, h)$, as defined in Lemma \ref{lem:EqualAsFixed} by $x\mapsto h^{-1}(g(x))$, is a retract of $\SD(g, h)$.
To see this, firstly note that $\SD(g, h)$ is a retract of $F(\Sigma)$, by Lemma \ref{lem:retractSD}, so we have that $g(\SD(g, h))$ is a retract of $\im(g)$, and so, by transitivity of retracts, $g(\SD(g, h))$ is a retract of $\langle \im(g)\cup\im(h)\rangle$.
As $g(\SD(g, h))\leq h(\SD(g, h))$, by Lemma \ref{lem:CompareImages}, this means that $g(\SD(g, h))$ is a retract of $h(\SD(g, h))$.
As $h$ is injective, this gives us that $h^{-1}(g(\SD(g, h)))=\phi_{(g, h)}(\SD(g, h))$ is a retract of $\SD(g, h)$, as claimed.

Our second fact, which is immediate from the definitions, is that the stable image $\phi_{(g, h)}^{\infty}:=\cap_{i=1}^{\infty}\phi_{(g, h)}^{i}$ is the stable domain of the identity function $\operatorname{id}\in\emo(\SD(g, h))$ with $\phi_{(g, h)}$, that is $\phi_{(g, h)}^{\infty}=\SD(\operatorname{id}, \phi_{(g, h)}).$

Our third fact is that the equaliser $\eq(g, h)$ is equal to the fixed subgroup $\fix(\phi_{(g, h)}|_{\phi_{(g, h)}^{\infty}})$, as
by applying Lemma \ref{lem:EqualAsFixed} twice we get
\[
\eq(g, h)=\fix(\phi_{(g, h)})=\eq(\operatorname{id}, \phi_{(g, h)})=\fix(\phi_{(g, h)}|_{\phi_{(g, h)}^{\infty}}).
\]

The algorithm to compute a basis for $\eq(g, h)$ is as follows.
Firstly, use Lemma \ref{lem:algorithmicSD} to compute a basis $\mathcal{B}_1$ for $\SD(g, h)$.
Then compute the map $\phi_{(g, h)}$, which we may do as for each $x\in\mathcal{B}_1$ we have $\phi_{(g, h)}(x)=h^{-1}(g(x))$.
Next, use Lemma \ref{lem:algorithmicSD} to compute a basis $\mathcal{B}_2$ for $\SD(\operatorname{id}, \phi_{(g, h)})=\phi_{(g, h)}^{\infty}$, which we may do as $\im(\phi_{(g, h)})$ is a retract of $\SD(g, h)$.
Now, $\phi_{(g, h)}$ acts as an automorphism on $\phi_{(g, h)}^{\infty}$, and so use the known algorithm in this setting to compute a basis $\mathcal{B}_3$ for $\fix(\phi_{(g, h)}|_{\phi_{(g, h)}^{\infty}})$ \cite{Bogopolski2016algorithm}.
Output this set $\mathcal{B}_3$ as a basis for $\eq(g, h)$.
\end{proof}

We now generalise Theorem \ref{thm:RetractAlgorithmic} to sets of homomorphisms, and to do this we define a new graph: For a set $S: F(\Sigma)\rightarrow F(\Delta)$ of homomorphisms, $\Gamma_S^r$ is the graph with vertex set $S$, with an edge connecting $g, h\in S$ if $\im(g)$ and $\im(h)$ are retracts of $\langle\im(g)\cup\im(h)\rangle$.

\begin{corollary}
\label{thm:RetractAlgorithmicSETS}
There exists an algorithm with input a finite set of injective homomorphisms $S: F(\Sigma)\rightarrow F(\Delta)$ such that the graph $\Gamma_S^r$ is connected, and output a basis for $\eq(S)$.
\end{corollary}

\begin{proof}
As $\Gamma_S^r$ is connected, $\eq(S)$ is the intersection of those equalisers $\eq(g, h)$ such that there is a edge connecting $g$ and $h$.
For any such connected pair $g, h\in S$, if $\im(g)\neq\im(h)$ then we can compute a basis for $\eq(g, h)$ by Theorem \ref{thm:RetractAlgorithmic}, while if $\im(g)=\im(h)$ then $\eq(g, h)=\fix(gh^{-1})$, and we can compute a basis using existing algorithms \cite{Bogopolski2016algorithm}.
As $S$ is finite, we can then compute a basis for $\eq(S)$ \cite[Corollary 9.5]{Kapovich2002stallings}, and the result follows.
\end{proof}

%%%----------------------------------------------%%%
%%%------------Inertly-induced maps---------%%%
%%%----------------------------------------------%%%

\section{Inertly induced maps}
\label{sec:induced}

Recall from the introduction that a pair of homomorphisms $g, h: F(\Sigma)\rightarrow F(\Delta)$ is \emph{inertly induced} if the pair can be viewed as the restrictions of a pair of homomorphisms $g', h': F(\Sigma')\rightarrow F(\Delta')$ such that $\im(g')$ and $\im(h')$ are inert subgroups of $\langle \im(g')\cup\im(h')\rangle$; that is, if there exists embeddings $\iota: F(\Sigma)\hookrightarrow F(\Sigma')$ and $\tau: F(\Delta)\hookrightarrow F(\Delta')$ and a pair of homomorphisms $g', h': F(\Sigma')\rightarrow F(\Delta')$ such that $g'(\iota(x))=\tau(g(x))$ and $h'(\iota(x))=\tau(h(x))$ for all $x\in F(\Sigma)$, and such that $\im(g')$ and $\im(h')$ are inert subgroups of $\langle \im(g')\cup\im(h')\rangle$.
Similarly, a set of maps $S: F(\Sigma)\rightarrow F(\Delta)$ is \emph{$F_2$-induced} if the set can be viewed as the restrictions of a set of homomorphisms $S': F(a', b')\rightarrow F(\Delta')$. Note that pairs of $F_2$-induced maps are inertly induced, as subgroups of rank two are inert in free groups.

\begin{example}
\label{ex:InertlyInduced}
The pair $g, h: F\{x, y, z\}\rightarrow F\{a,b\}$ defined by $g: x\mapsto a^4, y\mapsto a^{-1}b^2a, z\mapsto aba$ and $h: x\mapsto b^2, y\mapsto a^6, z\mapsto ba^3$ is induced by the pair $g', h':F\{x', y'\}\rightarrow F\{a, b\}$ defined by $g': x'\mapsto a^2, y'\mapsto a^{-1}ba$ and $h': x'\mapsto b, y'\mapsto a^3$, under the embedding $\iota: x\mapsto (x')^2, y\mapsto (y')^2, z\mapsto x'y'$. Therefore, the pair $g, h$ is $F_2$-induced, and hence inertly induced.
\end{example}

We now prove Corollary \ref{corol:InertiaInduced} from the introduction.

\begin{proof}[Proof of Corollary \ref{corol:InertiaInduced}]
We have that $\iota(\eq(g, h))=\iota(F(\Sigma))\cap\eq(g', h')$. As $\eq(g', h')$ is inert, by Theorem \ref{thm:Inert}.\ref{Inert:2}, we have that $\rk(\iota(\eq(g, h)))\leq\rk(\iota(F(\Sigma)))$. The result then follows as $\iota$ is injective.
\end{proof}

We now prove Corollary \ref{corol:F2Induced} from the introduction.

\begin{proof}[Proof of Corollary \ref{corol:F2Induced}]
We have that $\iota(\eq(S))=\iota(F(\Sigma))\cap\eq(S')$. As $\eq(S')$ has rank at most two, by Theorem \ref{thm:rankSETS}, it is inert. Hence, $\rk(\iota(\eq(S)))\leq\rk(\iota(F(\Sigma)))$. The result then follows as $\iota$ is injective.
\end{proof}

%%%----------------------------------------------%%%
%%%------------Concluding remarks-----------%%%
%%%----------------------------------------------%%%
\section{More on stable domains}
\label{sec:Concluding}
The stable domain of a pair of maps played a central role in this article, and we include here a brief discussion about this object.

\p{Symmetry}
As we saw in Example \ref{ex:symmetry}, $\SD(g, h)\neq\SD(h, g)$ in general, and so the stable domain is not a symmetric construction. We now characterise those injective maps for which $\SD(g, h)=\SD(h, g)$. Recall the maps $\phi_{(g, h)}\in\emo(\SD(g, h))$ and $\phi_{(h, g)}\in\emo(\SD(h, g))$ from Lemma \ref{lem:EqualAsFixed}.

\begin{proposition}
\label{prop:SDsymmetry}
Let $g, h: F(\Sigma)\rightarrow F(\Delta)$ be injective homomorphisms. Then $\SD(g, h)=\SD(h, g)$ if and only if $\phi_{(g, h)}\in\aut(\SD(g, h))$ and $\phi_{(h, g)}\in\aut(\SD(h, g))$.
\end{proposition}

\begin{proof}
Suppose $\SD(g, h)=\SD(h, g)$. By Lemma \ref{lem:CompareImages}, $g(\SD(g, h))=h(\SD(g, h))$ and so the monomorphism $\phi_{(g, h)}:x\mapsto h^{-1}g(x)$ is surjective, and so is an automorphism. Symmetrically, $\phi_{(h, g)}\in\aut(\SD(h, g))$ as required.

If $\phi_{(g, h)}\in\aut(\SD(g, h))$ then $h^{-1}(g(\SD(g, h))= \SD(g, h)$. Therefore, $g(\SD(g, h))= h(\SD(g, h))$ and so, by Lemma \ref{lem:CompareImages}, we have that $\SD(g, h)\leq\SD(h, g)$. Symmetrically, if $\phi_{(h, g)}\in\aut(\SD(h, g))$ then $\SD(h, g)\leq\SD(g, h)$. Hence, $\SD(g, h)=\SD(h, g)$ as required.
\end{proof}

\p{Finite generation}
By Theorem \ref{thm:EqRankBoundedSD}, it is important to understand the rank of the stable domain $\SD(g, h)$. Unfortunately, Example \ref{ex:finiteGen} showed that stable domains are not necessarily finitely generated.
The key point used in Example \ref{ex:finiteGen} was the non-injectivity of the map $g$.
This suggests the following question.
\begin{question}
\label{Qn:bases}
Suppose both $g$ and $h$ are injective.
\begin{enumerate}[label=(\alph*)]
\item Is $\SD(g, h)$ finitely generated?
\item If so, is $\rk(\SD(g, h))\leq|\Sigma|$?
\end{enumerate}
\end{question}
By Theorem \ref{thm:EqRankBoundedSD}, if $\rk(\SD(g, h))\leq|\Sigma|$ then $\rk(\eq(g, h))\leq|\Sigma|$, which would resolve the Equaliser Conjecture for injective maps. In fact, this would also resolve Conjecture \ref{Qn:Inert} for injective maps (see \ref{appendix:EquivalentConj}).

\p{Computing bases}
Ventura asked if there exists an algorithm to compute a basis for the stable image of a free group endomorphism \cite[Problem 4.6]{Dagstuhl2019}. As stable domains generalise stable images, the following question generalises Ventura's question.
\begin{question}
\label{Qn:algorithm}
Does there exist an algorithm with input a pair of injective homomorphisms $g, h: F(\Sigma)\rightarrow F(\Delta)$, and with output a finite basis for $\SD(g, h)$?
\end{question}

Recall from the introduction that Stallings asked if there exists an algorithm to compute a basis for $\eq(g, h)$, $g$ and $h$ as in Question \ref{Qn:algorithm}.
Lemma \ref{lem:algorithmicSD} answered Question \ref{Qn:algorithm} for certain maps, and this was the key step in proving a special case of Stallings' algorithmic question (Theorem \ref{thm:RetractAlgorithmic}).
This connection generalises, as a positive answer to Question \ref{Qn:algorithm} yields a positive answer to Stallings' algorithmic question: Firstly compute a basis for $\SD(g, h)$, and use this basis to describe the endomorphism $\phi_{(g, h)}: \SD(g, h)\rightarrow \SD(g, h)$. We can compute a basis for the stable image $\phi_{(g, h)}^{\infty}$ of $\phi_{(g, h)}$ (as stable images are themselves stable domains). As $\phi_{(g, h)}$ acts as an automorphism on $\phi_{(g, h)}^{\infty}$, we can compute a basis for the corresponding fixed subgroup $\fix(\phi_{(g, h)}|_{\phi_{(g, h)}^{\infty}})$ \cite{Bogopolski2016algorithm}. This subgroup is precisely $\fix(\phi_{(g, h)})$ \cite{Imrich1989Endomorphisms}, which, by Lemma \ref{lem:EqualAsFixed}, is $\eq(g, h)$ as required.

The above also allows one to compute a basis for $\fix(\phi)$, $\phi: F(\Sigma)\rightarrow F (\Sigma)$ any endomorphism; the case of $\phi$ injective follows immediately from the above, while if $\phi$ is non-injective then there exists a constructable injective endomorphism $\phi': F(\Sigma)\rightarrow F (\Sigma)$ and an constructable isomorphism $\pi: \fix(\phi)\rightarrow\fix(\phi')$ \cite{Imrich1989Endomorphisms}, and so a basis for $\fix(\phi)$ can be obtained by finding a basis for $\fix(\phi')$ and then reversing the isomorphism.

%%%----------------------------------------------%%%
%%%------------Pre-appendix technical TeX stuff-----------%%%
%%%----------------------------------------------%%%

%alter the section title to make clear its an appendix
\renewcommand{\thesection}{Appendix \Alph{section}}

%restart the section counter
\setcounter{section}{0}

%make the appendix proposition and corollary look nice.
\newtheorem{appendixProposition}{Proposition} 
\renewcommand{\theappendixProposition}{\Alph{section}.\arabic{appendixProposition}}
\newtheorem{appendixCorollary}[appendixProposition]{Corollary}

%%%----------------------------------------------%%%
%%%------------Appendix-----------%%%
%%%----------------------------------------------%%%

\section{The equivalence of Conjectures \ref{Qn:StallingsRank} and \ref{Qn:Inert}}
\label{appendix:EquivalentConj}
As we mentioned in the introduction, Ventura implied that the Equaliser Conjecture can be reformulated in terms of inertness, that is, Conjectures \ref{Qn:StallingsRank} and \ref{Qn:Inert} are equivalent. We prove this equivalence now, starting with the following general proposition, where the Equaliser Conjecture corresponds to part (\ref{EquivalentConj:1}) of the proposition and Conjecture \ref{Qn:Inert} corresponds to part (\ref{EquivalentConj:2}).

We say that a class $\mathcal{C}$ of free group homomorphisms is \emph{closed under restrictions} if for all maps $g:F(\Sigma)\rightarrow F(\Delta)$ in $\mathcal{C}$ and all finitely generated subgroups $K\leq F(\Sigma)$, the restriction map $g|_K: K\rightarrow F(\Delta)$, viewing $K$ as an abstract free group, is also contained in $\mathcal{C}$. For example, the classes of {all} free group homomorphisms and of all injective free group homomorphisms are closed under restrictions.

\begin{appendixProposition}
\label{prop:EquivalentConj}
Let $\mathcal{C}$ be a class of free group homomorphisms which is closed under restrictions.
The following are equivalent.
\begin{enumerate}
\item\label{EquivalentConj:1} For all free groups $F(\Sigma)$ and $F(\Delta)$ and all homomorphisms $g, h: F(\Sigma)\rightarrow F(\Delta)$ in $\mathcal{C}$ with $h$ injective, $\rk(\eq(g, h))\leq|\Sigma|$.
\item\label{EquivalentConj:2} For all free groups $F(\Sigma)$ and $F(\Delta)$ and all sets of homomorphisms $S: F(\Sigma)\rightarrow F(\Delta)$ with $S\subset\mathcal{C}$ and containing at least one injective map, $\eq(S)$ is an inert subgroup of $F(\Sigma)$.
\end{enumerate}
\end{appendixProposition}

\begin{proof}
Clearly (\ref{EquivalentConj:2}) implies (\ref{EquivalentConj:1}), as $\eq(g, h)=\eq(g, h)\cap F(\Sigma)$ so by inertness $\rk(\eq(g, h))\leq\rk(F(\Sigma))\leq|\Sigma|$.

For (\ref{EquivalentConj:1}) implies (\ref{EquivalentConj:2}), assume (\ref{EquivalentConj:1}) holds and consider a pair of homomorphisms $g, h: F(\Sigma)\rightarrow F(\Delta)$ with $h$ injective, and let $K$ be an arbitrary subgroup of $F(\Sigma)$. Note that if $K$ is not finitely generated then $\rk(\eq(g|_K, h|_K))\leq\rk(K)$, while if $K$ is finitely generated then the homomorphisms $g|_K, h|_K: K\rightarrow F(\Delta)$ satisfy part \ref{EquivalentConj:1} of the proposition (as $\mathcal{C}$ is closed under restrictions), and so by assumption $\rk(\eq(g|_K, h|_K))\leq\rk(K)$.
Then $\eq(g,h)\cap K=\eq(g|_K, h|_K)$ and so $\rk(\eq(g,h)\cap K)=\rk(\eq(g|_K, h|_K))\leq\rk(K)$. Therefore, as $K$ is arbitrary $\eq(g, h)$ is inert for all such maps $g$ and $h$.
Now consider a set of homomorphisms $S: F(\Sigma)\rightarrow F(\Delta)$ containing at least one injective map, $h$ say. Then $\eq(S)=\cap_{g\in S}\eq(g, h)$, and as each $\eq(g, h)$ is inert, and as $S$ is necessarily countable, the result follows by Lemma \ref{lem:InertiaIntersect}.
\end{proof}

As we noted above, the classes of {all} free group homomorphisms and of all injective free group homomorphisms are closed under restrictions. Therefore, the proposition has the following corollary.

\begin{appendixCorollary}\leavevmode
\begin{enumerate}
\item Conjecture \ref{Qn:StallingsRank} holds if and only if Conjecture \ref{Qn:Inert} holds.
\item Conjecture \ref{Qn:StallingsRank} holds for injective maps if and only if Conjecture \ref{Qn:Inert} holds for injective maps.
\end{enumerate}
\end{appendixCorollary}

\bibliographystyle{amsplain}
\bibliography{BibTexBibliography}

\end{document}